\documentclass[10pt, conference]{IEEEtran}

\usepackage{graphics} % for pdf, bitmapped graphics files
\usepackage{amsmath} % assumes amsmath package installed
\usepackage{amssymb}  % assumes amsmath package installed
\usepackage{amsthm}
\usepackage{bm}
\usepackage{soul}

\usepackage{acronym}
\usepackage{paralist}
\usepackage{float}
\usepackage{color}
\usepackage{epstopdf}
\usepackage{multicol}
\usepackage{tikz}
\usepackage{hyperref}

\newtheorem{Theorem}{Theorem}

\newtheorem{Problem}{Problem}

\newtheorem{Remark}{Remark}

\newtheorem{Proposition}{Proposition}

\newtheorem{Corollary}{Corollary}

\newtheorem{Definition}{Definition}

\makeatletter
\hypersetup{colorlinks=true}
\AtBeginDocument{\@ifpackageloaded{hyperref}
  {\def\@linkcolor{blue}
  \def\@anchorcolor{red}
  \def\@citecolor{red}
  \def\@filecolor{red}
  \def\@urlcolor{red}
  \def\@menucolor{red}
  \def\@pagecolor{red}
\begingroup
  \@makeother\`%
  \@makeother\=%
  \edef\x{%
    \edef\noexpand\x{%
      \endgroup
      \noexpand\toks@{%
        \catcode 96=\noexpand\the\catcode`\noexpand\`\relax
        \catcode 61=\noexpand\the\catcode`\noexpand\=\relax
      }%
    }%
    \noexpand\x
  }%
\x
\@makeother\`
\@makeother\=
}{}}
\makeatother

\IEEEoverridecommandlockouts
\begin{document}

% \title{\LARGE \bf Multi-Task Trajectory Planning: Ensuring Safety and Fixed-time Convergence with Quadratic Program-based Barriers}
\title{\LARGE \bf Control-Lyapunov and Control-Barrier Functions based Quadratic Program for Spatio-temporal Specifications}
\author{Kunal Garg  \and Dimitra Panagou
\thanks{The authors are with the Department of Aerospace Engineering, University of Michigan, Ann Arbor, MI, USA; \texttt{\{kgarg,dpanagou\}@umich.edu}.}
\thanks{The authors would like to acknowledge the support of the Air Force Office of Scientific
Research under award number FA9550-17-1-0284.}}

\maketitle

\begin{abstract}
This paper presents a method for control synthesis under spatio-temporal constraints. First, we consider the problem of reaching a set $S$ in a user-defined or prescribed time $T$. We define a new class of control Lyapunov functions, called prescribed-time control Lyapunov functions (PT CLF), and present sufficient conditions on the existence of a controller for this problem in terms of PT CLF. Then, we formulate a quadratic program (QP) to compute a control input that satisfies these sufficient conditions. Next, we consider control synthesis under spatio-temporal objectives given as: the closed-loop trajectories remain in a given set $S_s$ at all times; and, remain in a specific set $S_i$ during the time interval $[t_i, t_{i+1})$ for $i = 0, 1, \cdots, N$; and, reach the set $S_{i+1}$ on or before $t = t_{i+1}$. We show that such spatio-temporal specifications can be translated into temporal logic formulas. We present sufficient conditions on the existence of a control input in terms of PT CLF and control barrier functions. Then, we present a QP to compute the control input efficiently, and show its feasibility under the assumptions of existence of a PT CLF. To the best of authors' knowledge, this is the first paper proposing a QP based method for the aforementioned problem of satisfying spatio-temporal specifications for nonlinear control-affine dynamics with input constraints. We also discuss the limitations of the proposed methods and directions of future work to overcome these limitations. We present numerical examples to corroborate our proposed methods. 
\end{abstract}

\section{Introduction}

Driving the state of a dynamical system to a given desired set is an important problem, particularly in the fields of robot motion planning and safety-critical control. Various approaches have been developed in past to accomplish this task. Model predictive control (MPC)-based methods \cite{saska2014motion,grancharova2015uavs}, rapidly-exploring random tree (RRT) based methods \cite{svenstrup2010trajectory,palmieri2016rrt,salzman2016asymptotically}, and combinations of them \cite{svenstrup2010trajectory} have been studied extensively in the literature. In addition, Lyapunov-based methods, such as vector fields  \cite{kovacs2016novel,han2019robust} and control Lyapunov functions (CLF) \cite{li2018formally,srinivasan2018control,ames2012control} are also popular, in part because these methods are inherently amenable to Lyapunov-based analysis. Control design for systems with input and state constraints is not a trivial task, as these constraints impose limitations on several aspects of the control synthesis. For example, spatial constraints requiring the system trajectories to be in some safe set at all times are common in safety-critical applications. Furthermore, temporal constraints pertaining to convergence within a prescribed time appear in time-critical applications where completion of a task is required within a given time interval. \textit{Spatio-temporal} specifications impose spatial as well as temporal or time constraints on the system trajectories.

\begin{figure}[!ht]
    \centering
    \includegraphics[ width=1\columnwidth,clip]{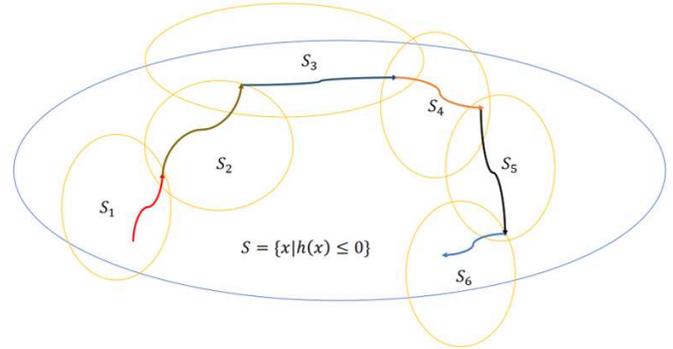}
    \caption{Motivating problem: The system trajectories need to visit the sets $S_i$, $i=1,\dots,6$ (orange regions) in a given time sequence, while always remaining in the set $S$ (blue region).}\label{fig: intro image}
\end{figure}

From practical point of view, considering safety constraints, e.g., avoiding collisions in multi-robot systems, avoiding static and dynamic obstacles, and in general, avoiding the unsafe regions in the state space, is crucial.  One of the most common methods of incorporating such spatial constraints on the system states is based on control barrier functions (CBF) \cite{ames2017control}. Barrier functions are used for the synthesis of safe controllers \cite{ames2017control,xu2015robustness} and barrier certificates are used as a verification tool to guarantee that the closed-loop trajectories remain safe at all times. The authors in \cite{barry2012safety} present sufficient conditions in terms of existence of a barrier certificate for forward-invariance of a given set, and propose a sum-of-squares formulation to find a Barrier certificate. In order to guarantee safety and convergence, a combination of CLFs and CBFs is used for control design \cite{ames2017control,ames2014control,romdlony2016stabilization}. In the CLF-CBF based controller, convergence is guaranteed due to the CLF and safety is guaranteed due to CBF.
\cite{tee2009barrier} utilizes Lyapunov-like barrier functions to guarantee asymptotic tracking of a time-varying output trajectory, while the system output always remains inside a given set. The authors in \cite{ames2017control,ames2014control} present conditions using zeroing barrier functions so that the set defined as $\mathcal C = \{x\;|\; h(x)\geq 0\}$, where $h(x)$ is a user-defined smooth function, is forward invariant. 

More recently, quadratic program (QP) based approaches have gained popularity for control synthesis; with this approach, the CLF and CBF conditions are formulated as inequalities that are linear in the control input \cite{li2018formally,srinivasan2018control,ames2017control,rauscher2016constrained}. These methods are suitable for real-time implementation as QPs can be solved very efficiently. The authors in \cite{ames2017control} combine the control performance objectives and safety objectives, represented using CLF and CBF, respectively, via a single QP. Authors in \cite{lindemann2019control} use CBF to encode signal-temporal logic (STL) based specifications and formulate a QP to compute the control input (see \cite{lindemann2019control} for details on STL-based  specifications for robot motion planning). The aforementioned work \cite{kovacs2016novel,han2019robust}, \cite{ames2012control}-\cite{rauscher2016constrained} concerns with designing control laws so that the reachability objectives, such as reaching a desired location or a desired goal set, are achieved as time goes to infinity, i.e., asymptotically. 

In contrast to asymptotic stability (AS), which pertains to convergence as time goes to infinity, finite-time stability (FTS) is a concept that guarantees convergence of solutions in finite time. In the seminal work \cite{bhat2000finite}, the authors introduce the necessary and sufficient conditions in terms of Lyapunov functions under which continuous, autonomous systems exhibit FTS. The authors in \cite{li2018formally} formulate a QP to ensure finite-time convergence of the closed-loop trajectories to a set $S = \{x\; |\; h(x)\leq 0\}$ with input constraints. Fixed-time stability (FxTS) \cite{polyakov2012nonlinear} is a stronger notion than FTS, where the time of convergence does not depend upon the initial conditions. More recently, the authors in \cite{holloway2019prescribed} used the notion or prescribed-time or user-defined time stability, where the time of convergence can be chosen by the user \textit{a priori}.

% {\color{red}\st{In this paper, we first study the problem of reaching a given set in a user-defined time $T$ for a general class of control-affine systems with input constraints. We extend the proposed formulation to guarantee that once the trajectories reach the desired set $S$, they stay in the set $S$ for all future times. }}
% One of the main drawbacks of the method proposed in \cite{srinivasan2018control,li2018formally} is that they use the function $h(x)$ itself as a CLF. 
In most of the  aforementioned work, only one safety and one convergence objectives are considered. In this paper, we consider a multi-task problem of designing a control input. The considered objectives are of the following form: (i) the system trajectories should stay in a given set $S_s$ at all times, (ii) the system trajectories should stay in a set $S_i$ in the time-interval $[t_i, t_{i+1})$ for $i = 0, 1, 2, \cdots, N$, where $\{t_0, t_1,\cdots, t_N,\}$ is a user-defined time sequence, (iii) the trajectories should reach the set $S_{i+1}$ before time instant $t = t_{i+1}$, and (iv) the control input should satisfy control constraints at all times. We show that such spatio-temporal specifications can be translated into a STL formula. In \cite{lindemann2019control}, the authors consider the problem of generating controller to satisfy STL specifications under the assumption that the system dynamics are equivalent to a single-integrator dynamics. To the best of authors' knowledge, this is the first paper proposing a QP based method for the aforementioned problem of satisfying spatio-temporal specifications without making any assumptions on the system dynamics. We first study the problem of reaching a given set in a user-defined time $T$ for a general class of control-affine systems with input constraints. We extend the proposed formulation to guarantee that once the trajectories reach the desired set $S$, they stay in the set $S$ for all future times. We define the new notion of prescribed-time CLF (PT CLF), and use it to solve the problem of reaching a given goal set within a given prescribed time $T$. Then, we present sufficient conditions in terms PT CLF and CBF to guarantee that the given spatio-temporal specifications are met. Finally, we present a QP-based optimization problem that can compute a control input for the same, and show its feasibility under some mild conditions. In contrast to earlier work \cite{li2018formally,srinivasan2018control,ames2017control,lindemann2019control}, our proposed framework is able to accommodate spatio-temporal, i.e., both state and time, constraints in the presence of control input constraints. Furthermore, in contrast to the results in \cite{li2018formally,li2019finite}, where under the traditional notion of FTS, as defined in \cite{bhat2000finite}, the convergence time depends upon the initial conditions, the closed-loop system trajectories resulting from our controller reach the given set in a prescribed time that can be chosen arbitrarily and independently of the initial conditions.
% Finally, we discuss the limitations of the proposed method (which are similar to the earlier work \cite{li2018formally,ames2017control}). The main limitations of using the function $h(x)$, which defines the set $S$, as the CLF or CBF in control synthesis is that it requires that $h(x)$ is continuously differentiable \textbf{Okay but there is a whole field of nonsmooth analysis and nonsmooth Lyapunov functions already developed for treating stability with nonsmooth tools -- you may want to say that in future work we will also consider nonsmooth functions? and also, alternatively, we will investigate whether the problem can be potentially solved as you discuss in the discussion section. Also, this is too many conjectures to be put forward in the introduction, perhaps you can save them for the discussion section, i.e., just mention that in the discussion section we summarize our thoughts on how to further generalize to nonsmooth functions etc}, 

% and that it satisfies the CLF/CBF conditions for the given system dynamics. At the end of this paper, we discuss a method of overcoming this drawback by giving sufficient conditions to compute the CLF and the control input together. 

The rest of the paper is organized as follows: In Section \ref{sec: math}, we present the notations used in the paper and background material on  various notions of finite-time stability. In Section \ref{sec: FxTS S}, we study the problem of reaching a set $S$ in a prescribed time and staying there for all future times. In Section \ref{sec: multiple FxTS}, we consider the general multi-task problem for spatio-temporal specifications and formulate a QP to find the control input. We show three numerical examples in Section \ref{sec: simulations} to corroborate our theoretical results. We discuss the limitations of the proposed methods and propose a direction to relax the assumptions used in deriving the main results and summarize our thoughts on future work in Section \ref{sec: discussion}. Finally, we present the conclusions in Section \ref{sec: conclusion}. 

\section{Mathematical Preliminaries}\label{sec: math}
    
\subsection{Notation}
$\mathbb R$ denotes the set of reals and $\mathbb R_+$ denotes the set of non-negative reals. The boundary of a closed set $S$ is denoted by $\partial S$ and its interior by $\textrm{int} (S) \triangleq S\setminus\partial S$. The Lie derivative of a function $h:\mathbb R^n\rightarrow\mathbb R$ along a vector field $f:\mathbb R^n\rightarrow\mathbb R^n$ is denoted as $L_fh \triangleq \frac{\partial h}{\partial x}f$. We use $\|x\|_p$ to denote the $p$-norm of the vector $x\in \mathbb R^n$ and simply use $\|\cdot\|$ to denote the Euclidean norm. 
% For any set $S$, the distance of point $x\notin S$ from $S$ is denoted as $\|x\|_S = \inf_{y\in S}\|x-y\|$. 
% A function $\alpha:\mathbb R_+\rightarrow \mathbb R_+$ is said to belong to class-$\mathcal K$, denoted as $\alpha\in \mathcal K$, if $\alpha(0) = 0$ and $\alpha$ is strictly increasing.

\subsection{Preliminaries}
Consider the system: 
\begin{align}\label{ex sys}
\dot x(t) = f(x(t)),
\end{align}
where $x\in \mathbb R^n$ and $f: \mathbb R^n \rightarrow \mathbb R^n$ is continuous with $f(0)=0$. As defined in \cite{bhat2000finite}, the origin is said to be an FTS equilibrium of \eqref{ex sys} if it is Lyapunov stable and \textit{finite-time convergent}, i.e., for all $x(0) \in \mathcal N \setminus\{0\}$, where $\mathcal N$ is some open neighborhood of the origin, $\lim_{t\to T} x(t)=0$, where $T = T(x(0))<\infty$, depends upon the initial condition $x(0)$. The authors in \cite{polyakov2012nonlinear} presented the following result for FxTS, where the time of convergence does not depend upon the initial condition.

\begin{Theorem}[\cite{polyakov2012nonlinear}]\label{FxTS TH}
Suppose there exists a positive definite function $V$ for system \eqref{ex sys} such that 
\begin{align}
    \dot V(x) \leq -(aV(x)^p+bV(x)^q)^k,
\end{align}
with $a,b,p,q,k>0$, $pk<1$ and $qk>1$. Then, the origin of \eqref{ex sys} is FxTS with continuous settling time function 
\begin{align}
    T \leq \frac{1}{a^k(1-pk)} + \frac{1}{b^k(qk-1)}. 
\end{align}
\end{Theorem}
If the settling-time $T$ can be chosen a priori by the user, then the origin is called as user-defined or prescribed-time stable \cite{holloway2019prescribed}.

\section{Prescribed-time Set Reachability}\label{sec: FxTS S}
In this section, we consider the problem of reaching a set $S = \{x\; |\; h(x)\leq 0\}$ in a user-defined or prescribed time $T$, where $h:\mathbb R^n\rightarrow\mathbb R$ is a user-defined function. Consider the system 
\begin{align}\label{cont aff sys}
    \dot x(t) = f(x(t)) + g(x(t))u, \; x(t_0) = x_0, 
\end{align}
where $x\in \mathbb R^n$ is the state-vector, $f:\mathbb R^n\rightarrow \mathbb R^n$, $g:\mathbb R^n\rightarrow\mathbb R^{n\times m}$ are system vector fields, and $u\in \mathbb R^m$ is the control input. The problem statement can be formally written as:

\begin{Problem}\label{P reach S}
Design a control input $u(t)\in \mathcal U = \{v\; |\; A_uv\leq b_u\}$, so that the closed-loop trajectories of \eqref{cont aff sys} reach the set $S = \{x\; |\; h(x)\leq 0\}$ in a prescribed time $T$, where $h(x)$ is a user-defined continuously differentiable function, and $A_u\in \mathbb R^{l\times m}, b_u\in\mathbb R^l$ are user-defined matrices. 
\end{Problem}

Input constraints of the form $u(t)\in \mathcal U = \{v\; |\; A_uv\leq b_u\}$ are very commonly considered in the literature \cite{ames2017control}. Now, we present sufficient conditions for existence of a control input $u$ that solves Problem \ref{P reach S}. First, we define a new class of CLF with prescribed-time convergence guarantees:

\begin{Definition}\label{def: CLF PT}
\textbf{PT CLF-$S$}: A continuously differentiable function $V:\mathbb R^n\rightarrow\mathbb R$ is called PT CLF-$S$ for \eqref{cont aff sys} with parameters $a_1,a_2,b_1,b_2$, if it is positive definite with respect to the set $S$, i.e., $V(x) >0$ for all $x\notin S$, $V(x) = 0$ for all $x\in \partial S$,  and the following holds: 
\begin{align}\label{fxts h ineq def}
    \inf_{u\in \mathcal U}\{L_fV+L_gVu\}\leq -a_1V^{b_1}-a_2V^{b_2},
\end{align}
for all $x\notin \textrm{int}(S)$, where $a_1, a_2>0$, $b_1>1$ and $0<b_2<1$.
satisfy
\begin{align}
     \frac{1}{a_1(b_1-1)} + \frac{1}{a_2(1-b_2)}\leq T,
\end{align}
where $T>0$ is the prescribed time.
\end{Definition}

\noindent Definition \ref{def: CLF PT} provides a CLF that guarantees convergence of the solutions to the origin within prescribed time $T$. Note that the traditional notions of CLF \cite{li2018formally} and exponential CLF \cite{ames2014rapidly} are special cases of Definition \ref{def: CLF PT}, with $a_1 = a_2 = 0$, and $a_2 = 0$, $b_1 = 1$, respectively. Based on this definition, we can readily state the following result. 

% \begin{Theorem}\label{FxTS reach set S}
% If there exist constants $\alpha_1, \alpha_2>0$, $\gamma_1>1$ and $0<\gamma_2<1$, satisfying 
% \begin{align}\label{T a1 a2 set S}
%      T\geq \frac{1}{\alpha_1(\gamma_1-1)} + \frac{1}{\alpha_2(1-\gamma_2)},
% \end{align}
% such that the following holds
% \begin{align}\label{fxts h ineq}
%     \inf_{u\in \mathcal U}\{L_fh(x) + L_gh(x)u+\alpha_1h(x)^{\gamma_1}+\alpha_2h(x)^{\gamma_2}\}\leq 0,
% \end{align}
% for all $x\notin S$, then there exists $u(t)\in \mathcal U$, such that the closed-loop trajectories of \eqref{cont aff sys} reach the set $S$ within fixed time $T$ for all initial conditions $x(0)\notin S$.
% \end{Theorem}

\begin{Theorem}\label{FxTS reach set S}
If there exist constants $\alpha_1, \alpha_2>0$, $\gamma_1>1$ and $0<\gamma_2<1$, satisfying 
\begin{align}\label{T a1 a2 set S}
    \frac{1}{\alpha_1(\gamma_1-1)} + \frac{1}{\alpha_2(1-\gamma_2)}\leq T,
\end{align}
such that $h$ is PT CLF-$S$ with parameters $\alpha_1,\alpha_2,\gamma_1,\gamma_2$, then there exists $u(t)\in \mathcal U$, such that the closed-loop trajectories of \eqref{cont aff sys} reach the set $S$ within prescribed time $T$ for all initial conditions $x(0)\notin S$.
\end{Theorem}

\begin{proof}
Choose the candidate Lyapunov function $V(x) = h(x)$ for $x\notin S$. From the definition of set $S$, we know that $x\notin S$ implies $h(x)>0$, which implies $V(x)$ is positive definite with respect to the set $S$. Now, since \eqref{fxts h ineq def} holds for all $x\notin S$ for some $u\in \mathcal U$, we have that 
\begin{align*}
    \dot V = \dot h \leq-\alpha_1h^{\gamma_1}-\alpha_2h^{\gamma_2}.
\end{align*}
Hence, using Theorem \ref{FxTS TH}, we obtain that for all $t\geq T_0$, $V(x(t)) =  h(x(t))= 0$ where $T_0 \leq \frac{1}{\alpha_1(\gamma_1-1)} + \frac{1}{\alpha_2(1-\gamma_2)} \overset{\eqref{T a1 a2 set S}}{\leq} T$. This implies that the closed-loop trajectories reach the set $S$ within prescribed time $T$.  
\end{proof}

Theorem \ref{FxTS reach set S} deals with reaching the set $S$ before time $t = T$. Next, we present a result that guarantees that the closed-loop trajectories reach the set $S$ within a prescribed time $T$ and stay there for all future times using \eqref{eq: ZCBF}. 

\begin{Corollary}\label{Set S reach and inv}
If there exist constants $\alpha_1, \alpha_2>0$, $\gamma_1>1$ and $0<\gamma_2<1$ satisfying \eqref{T a1 a2 set S}, such that the following holds
\begin{align}\label{fxts h ineq new}
    \inf_{u\in \mathcal U}\{L_fh(x) + L_gh(x)u\}\leq & -\alpha_1\max\{0,h(x)\}^{\gamma_1}\nonumber\\
    & -\alpha_2\max\{0,h(x)\}^{\gamma_2},
\end{align}
for all $x$, then the closed-loop trajectories of \eqref{cont aff sys} reach the set $S$ within prescribed time $T<\infty$ for all initial conditions $x(0)\in \mathbb R^n$, and stay there for all future times.
\end{Corollary}
\begin{proof}
Note that once the trajectories of \eqref{cont aff sys} reach the set $S$, we have $h(x) = 0$. From \eqref{fxts h ineq new}, we obtain that for $h(x) = 0$, $\dot h(x)\leq 0$. Hence, $h(x)$ is non-increasing on the boundary of the set $S$, and hence, the set $S$ is forward invariant under the control input $u$ satisfying \eqref{fxts h ineq new}. So, the closed-loop trajectories stay in the set $S$ once they reach the set $S$. 
\end{proof}

As pointed out in \cite{li2018formally}, QPs can be solved very efficiently and can be used for real-time implementation. So, we present a QP-based formulation to compute the control input that satisfies the conditions of Corollary \ref{Set S reach and inv}. 

\begin{Theorem}\label{QP P1 Theorem}
Let the solution to the following QP 
\begin{subequations}\label{QP u P1}
\begin{align}
\min_{v, \alpha_1, \alpha_2} \;& \frac{1}{2}\|v\|^2 \\
    \textrm{s.t.} \; L_fh(x) + L_gh(x)v  & \leq  -\alpha_1\max\{0,h(x)\}^{\gamma_1}\nonumber\\
    & -\alpha_2\max\{0,h(x)\}^{\gamma_2}\\
    \frac{2}{T}\leq  \alpha_1&(\gamma_1-1), \label{a1 T const 1} \\
    \frac{2}{T}\leq \alpha_2 & (1-\gamma_2), \label{a2 T const 1} \\
    A_uv & \leq b_u,
\end{align}
\end{subequations}
where $\gamma_1>1$ and $0<\gamma_2<1$, is denoted as $[\bar \alpha_1 \; \bar \alpha_2 \; \bar v(t)]$. Then, the control input defined as $u(t) =\bar v(t)$ satisfies \eqref{fxts h ineq new}, and $\alpha_1 = \bar\alpha_1, \alpha_2 = \bar \alpha_2$ satisfy \eqref{T a1 a2 set S}. 
\end{Theorem}
\begin{proof}
First, note that the optimization variables in \eqref{QP u P1} are $\alpha_1, \alpha_2$ and $v$. The objective of the optimization problem \eqref{QP u P1} is quadratic in $v$ and the constraints are linear in the optimization variables. Hence, \eqref{QP u P1} is a QP. Now, first constraints of \eqref{QP u P1} is equivalent to \eqref{fxts h ineq new}. Constraints \eqref{a1 T const 1}-\eqref{a2 T const 1} make sure that $\alpha_1, \alpha_2$ are positive and the time constraint \eqref{T a1 a2 set S} is satisfied:
\begin{align*}
\frac{1}{\alpha_1(\gamma_1-1)} + \frac{1}{\alpha_2(1-\gamma_2)} & \overset{\eqref{a1 T const 1}-\eqref{a2 T const 1}}{\leq} T.
\end{align*}
The last constraint in \eqref{QP u P1} implies that $u\in \mathcal U$. Hence, the solution to \eqref{QP u P1} satisfies \eqref{T a1 a2 set S} and \eqref{fxts h ineq new}.
\end{proof}

% \begin{Remark}
% In contrast to \cite{li2018formally,srinivasan2018control}, our method has two advantages. First, in the aforementioned work, the time of convergence $T$ depends upon the initial conditions $x(0)$, and grows larger as the distance of $x(0)$ increases from the set $S$; while in our formulation, it is independent of $x(0)$. Second, previous work only concerns about reaching the set $S$, while the formulation in Theorem \ref{QP P1 Theorem} additionally renders the set $S$ forward invariant.  
% \end{Remark}

\section{Control Synthesis for STL specifications}\label{sec: multiple FxTS}
\subsection{Problem formulation}
In this section, we consider a general problem of designing control input for \eqref{cont aff sys} such that the closed-loop trajectories satisfy spatio-temporal specifications defined as follows. Let $h_i(x)$ be the function defining the set $S_i = \{x\; |\; h_i(x)\leq 0\}$ for $i\in \Sigma = \{0,1, 2, \cdots, N\}$ such that $S_i\bigcap S_{i+1}\neq \emptyset$ for all $0\leq i\leq N-1$. Let $S_s = \{x\; |\; h(x)\leq 0\}$ be such that $S_s\bigcap S_0\neq \emptyset$. Let $[t_0, t_1), [t_1, t_2), \cdots, [t_N, t_{N+1})$ be the set of intervals such that $t_{i+1}-t_i\geq \bar T$ for some $0<\bar T<\infty$, for all $0\leq i\leq N-1$. Assume that the functions $h(x), h_i(x)$ are continuously differentiable. We consider the following problem. 

\begin{Problem}\label{P S Si}
Assume $x(t_0)\in S_0\bigcap S_s$. Design a control input $u(t)\in \mathcal U = \{u\; |\; A_uu\leq b_u\}$, so that the closed-loop trajectories satisfy the following for all $i\in \Sigma:$
\begin{subequations}\label{P1 const}
\begin{align}
    x(t)\in S_s\; & \forall \; t\geq t_0,\\
    x(t)\in S_i\; & \forall\; t\in [t_i, t_{i+1}).\label{s_i inv constr}
\end{align}
\end{subequations}
\end{Problem}
Note that \eqref{s_i inv constr} inherently requires that $x(t_{i+1})\in S_{i+1}$ for $1\leq i\leq N-1$, i.e., the trajectories should reach the set $S_{i+1}$ on or before $t = t_{i+1}$, while staying in the set $S_i$ for all times $t\in [t_i, t_{i+1})$. Problem \ref{P S Si} can be readily translated into temporal logic formulas for the form of specifications that are encountered, for instance, in mission planning problems. The STL specifications, given by formula $\phi$ include the following semantics (see \cite{lindemann2019control} for more details):
\begin{itemize}
    \item $(x,t)\models\phi \iff h(x(t))\leq 0$;
    \item $(x,t)\models\lnot \phi \iff h(x(t)) > 0$;
    \item $(x,t)\models \phi_1\land \phi_2 \iff (x,t)\models \phi_1\land (x,t)\models \phi_2$;
    \item $(x,t)\models G_{[a,b]}\phi \iff h(x(t))\leq 0, \forall t\in [a,b]$;
    \item $(x,t)\models F_{[a,b]}\phi \iff \exists t\in [a,b]$ such that $h(x(t))\leq 0$,
\end{itemize}
where $\phi = \textrm{true}$ if $h(x)\leq 0$ and $\phi = \textrm{false}$ if $h(x)>0$. So, Problem \ref{P S Si} can be written in the STL semantics as follows.
\begin{Problem}
Design control input $u\in \mathcal U$ so that the closed-loop trajectories satisfy 
\begin{align}
   (x,t)\models & G_{[t_0, t_N]}\phi_s\land G_{[t_0, t_1]}\phi_0\land F_{[t_0, t_1]}\phi_1\land G_{[t_1, t_2]}\phi_1 \nonumber\\
    & \land F_{[t_1, t_2]}\phi_2\land \cdots \land G_{[t_{N-1},t_N]}\phi_{N-1}\land F_{[t_{N-1}, t_N]}\phi_N,
\end{align}
where $\phi\;(\textrm{respectively},\; \phi_i) = \textrm{true}$ if $h(x)\;(\textrm{respectively},\; h_i(x))\leq 0$, and false otherwise.
\end{Problem}

\begin{Remark}
If the STL-based specifications satisfy certain assumptions, then these specifications can be posed as an instance of Problem \ref{P S Si}. For illustration, consider Example 2 from \cite{lindemann2019control}. The STL specification $\phi= \phi_1\land\phi_2$, where $\phi_1 = F_{[5,15]}(\|x-[10 \; 0]^T\|\leq 5)$ and $\phi_2 = G_{[5,15]}(\|x-[10 \; 5]^T\|\leq 10)$, means that the closed loop trajectories should reach the set $S_1 = \{x\; |\; \|x-[10 \; 5]^T\|\leq 10\}$ on or before $t = 5$ sec, remain in the set $S_1$ for $t\in [5, 15]$ and reach the set $S_2 = \{x\; |\; \|x-[10 \; 0]^T\|\leq 5\}$ on or before $t = 15$. Since $S_1\bigcap S_2\neq\emptyset$, we can use the problem set of Problem \ref{P S Si} to address these specifications. In Section \ref{sec: simulations}, we present an example on how to address problems that do not satisfy the setup of Problem \ref{P S Si}, i.e., if the functions $h(x)$ or $h_i(x)$ are non-smooth or $S_i\bigcap S_{i+1} = \emptyset$, e.g., the case study in \cite{lindemann2017robust}.
\end{Remark}

\subsection{Main results}
In this work, we use the conditions of zeroing CBF (ZCBF) to ensure safety or forward invariance of the safe set $S_s$. The ZCBF is defined by the authors in \cite{ames2017control} as following. 
\begin{Definition}
A continuously differentiable function $B:\mathbb R^n\rightarrow \mathbb R$ is called as ZCBF for \eqref{cont aff sys} for set $S_s$ if $B(x)<0$ for $x\in \textrm{int}(S_s)$, $B(x) = 0$ for $x\in \partial S_s$, and there exists a continuous, increasing function $\alpha:\mathbb R_+\rightarrow\mathbb R_+$, with $\alpha(0) = 0$, such that 
\begin{align}\label{eq: ZCBF}
    \inf_{u\in \mathcal U}\{L_fB(x)+L_gB(x)u\}\leq \alpha(-B(x)),
\end{align}
for all $x\in S_s$.
\end{Definition}
\noindent One special case of \eqref{eq: ZCBF} is 
\begin{align}\label{eq: ZCBF spec}
    \inf_{u\in \mathcal U}\{L_fB(x)+L_gB(x)u\}\leq -\rho B(x),
\end{align}
for some $\rho\in \mathbb R$. In \cite[Remark 6]{ames2017control}, the authors mention that $B$ is is a ZCBF if \eqref{eq: ZCBF spec} holds with $\rho>0$. We note that this restriction is not needed for guaranteeing safety. We present sufficient conditions in terms of PT CLF-ZCBF like inequalities for existence of control input $u$ that solves Problem \eqref{P S Si}.
\begin{Theorem}\label{P1 solve suff}
If there exist parameters $a_{i1}, a_{i2}$, $\gamma_{i1}>1$ and $0<\gamma_{i2}<1$ for $i\in \Sigma$ such that
\begin{align}\label{bar T const}
     \bar T\geq \max_{i\in \Sigma} \Big\{\frac{1}{a_{i1}(\gamma_{i1}-1)} + \frac{1}{a_{i2}(1-\gamma_{i2})}\Big\},
\end{align}
and a control input $u(t)$ such that the following holds
\begin{subequations}\label{u const P1}
\begin{align}
    & \inf_{u\in \mathcal U}\{L_fh(x) + L_gh(x)u \} \leq -\lambda_hh(x), \label{S inv}\\
    & \inf_{u\in \mathcal U}\{L_fh_i(x) + L_gh_i(x)u\}\leq -\lambda_ih_i(x) ,\label{Si inv const}\\
    & \inf_{u\in \mathcal U}\{L_fh_{i+1} + L_gh_{i+1}u\}\leq  -a_{i1}\max\{0,h_{i+1}\}^{\gamma_{i1}}\nonumber\\
     & \quad \quad \quad \quad \quad \quad \quad \quad \quad \quad \quad \;\; -a_{i2}\max\{0,h_{i+1}\}^{\gamma_{i2}}, \label{Si+1 FxTS reach}
\end{align}
\end{subequations}
for $t\in [t_i, t_{i+1})$, for each $i\in \Sigma$, then, under the effect of control input $u$, the closed-loop trajectories satisfy \eqref{P1 const}. 
\end{Theorem}
\begin{proof}
Since $x(t_0)\in S_s\bigcap S_0$, we have that $h(x(t_0))\leq 0$. 
Note that \eqref{S inv} is independent of $i$, i.e., it is needed that \eqref{S inv} holds for all $t\in [t_0, t_{N+1})$. If the control input satisfies \eqref{S inv}, then the set $S_s$ is forward invariant (Corollary \ref{Set S reach and inv}). 
% To see why this is true, consider the case when $x\in \textrm{bd}(S_s)$. This means that $h(x) = 0$. From \eqref{S inv}, we obtain that $ \inf_{u\in \mathcal U}\{L_fh(x) + L_gh(x)u\} = \inf_{u\in \mathcal U}\dot h \leq 0$, which implies $h(x)$ is non-increasing on the boundary of the set $S_s$. Hence, \eqref{S inv} implies that the set $S_s$ is forward invariant.
Similarly, using \eqref{Si inv const}, we conclude that for $t\in [t_0, t_1)$, the set $S_0$ is forward-invariant. Finally, for $x\notin S_1$, from \eqref{Si+1 FxTS reach}, we obtain $\dot h_{1} \leq -a_{01}h_{1}^{\gamma_{01}}-a_{02}h_{1}^{\gamma_{02}}$.
Using Theorem \ref{FxTS reach set S}, we obtain that the closed-loop trajectories satisfy $h_1(x(t)) = 0$ for $t\geq t_0+T_0$, where $T_0\leq \bar T$. 
Hence, we obtain that $ t_0+T_0\leq t_0+\bar T\leq t_1$, which implies that the closed-loop trajectories reach the set $S_1$ on or before $t=t_1$. Also, once trajectories reach the set $S_1$, we have that $\dot h_1 \leq 0$, i.e., the set $S_1$ is forward invariant. Hence, the closed-loop trajectories reach the set $S_1$ on or before $t = t_1$ and stay in the set $S_1$ till $t = t_1$. So, at $t = t_1$, we have $x(t_1)\in S_s\bigcap S_1$. 

Using the same arguments for each $i = 1, 2, \cdots, N-1$, we obtain that the closed-loop trajectories satisfy \eqref{P1 const} under the effect of control input $u$ satisfying \eqref{u const P1}. 
\end{proof}

Note that inequalities \eqref{S inv} and \eqref{Si inv const} are ZCBF conditions that render the set $S$ and $S_i$ forward-invariant, while \eqref{Si+1 FxTS reach} is the PT CLF condition that guarantees fixed-time convergence to set $S_{i+1}$, as well forward invariance of the set $S_{i+1}$ once trajectories reach the set $S_{i+1}$. 

\begin{Remark}
In contrast to \cite{lindemann2019control}, where the authors assume that $g(x)g(x)^T$ is positive definite, we do not make any assumptions on the system vector fields $f$ and $g$. In fact, for $m<n$, this condition is not satisfied for \eqref{cont aff sys}. Furthermore, \cite{lindemann2019control} does not consider any input constraints. 
\end{Remark}

Lastly, we formulate a QP based optimization problem in order to find the parameters $a_{i1}, a_{i2}, \lambda_h, \lambda_i$ for each $i\in \Sigma$ and the control input $u(t)$ so that \eqref{bar T const} and \eqref{u const P1} are satisfied. Consider the optimization problem 

\begin{subequations}\label{QP u P2}
\begin{align}
\min_{v, a_{i1}, a_{i2}, \lambda_h, \lambda_i} \; \frac{1}{2}\|v\|^2 & \\
    \textrm{s.t.} \; L_fh(x) + L_gh(x)v + &\lambda_hh(x) \leq 0,\label{safety set Sh}\\
    L_fh_i(x) + L_gh_i(x)v+ &\lambda_ih_i(x)\leq 0 ,\label{safety set Si}\\
    L_fh_{i+1} + L_gh_{i+1}v & \leq -a_{i1}\max\{0,h_{i+1}\}^{\gamma_{i1}}\nonumber\\
    & -a_{i2}\max\{0,h_{i+1}\}^{\gamma_{i2}},\label{conv const}\\
    A_uv\leq b_u, & \label{input bound const}\\
    \frac{2}{\bar T} \leq a_{i1} & (\gamma_{i1}-1), \label{a1 T const}\\
    \frac{2}{\bar T} \leq a_{i2} & (1-\gamma_{i2}), \label{a2 T const}
\end{align}
\end{subequations}
where $\gamma_{i1}>1$ and $0<\gamma_{i2}<1$. Let the solution to \eqref{QP u P2} is denoted as $[\bar a_{i1} \; \bar a_{i2} \; \bar \lambda_h\; \bar \lambda_i \;\bar v_i]$ for $t\in [t_i, t_{i+1})$, for $i\in \Sigma$. We can now state the main result of the paper. 

\begin{Theorem}
If the functions $h_i$ are PT CLF-$S_i$ for all $i\in \Sigma$, then, the solution to \eqref{QP u P2} exists, and the control input defined as 
\begin{align}\label{u input form}
    u(t) = \bar v_i(t), \; t\in [t_i, t_{i+1}), \; i\in \Sigma
\end{align}
satisfies \eqref{u const P1}, and $a_{i1} = \bar a_{i1}, a_{i2} = \bar a_{i2}$ satisfy \eqref{bar T const}. 
\end{Theorem}
\begin{proof}
First, note that the optimization variables in \eqref{QP u P2} are $a_{i1}, a_{i2}, \lambda_h, \lambda_i$ and $v$. The constraints are linear in these variables, while the objective function is quadratic in $v$. Hence, the optimization problem \eqref{QP u P2} is a QP. It is easy to show that the problem \eqref{QP u P2} is feasible if $h_{i+1}$ is a PT CLF-$S_{i+1}$ with respect to $a_{i1},a_{i2}$ satisfying \eqref{bar T const}. To see why this is true, note that there exists $v$ satisfying \eqref{conv const}-\eqref{input bound const}, since $h_{i+1}$ is a PT CLF-$S_{i+1}$. With this $v$, one can choose $\lambda_h, \lambda_i$ satisfying \eqref{safety set Sh}-\eqref{safety set Si}, respectively, and $a_{i1}, a_{i2}$ satisfying \eqref{a1 T const}-\eqref{a2 T const}, respectively. Hence, there exists a solution to the QP \eqref{QP u P2}. Note that the initial four constraint are equivalent to the three inequalities in \eqref{u const P1}. Next, \eqref{a1 T const}-\eqref{a2 T const} imply that $\frac{1}{a_{i1}(\gamma_{i1}-1)} +\frac{1}{a_{i2}(1-\gamma_{i2})}\leq \bar T$,
so, \eqref{bar T const} is also satisfied. Hence, with the last two constraints in \eqref{QP u P1}, all the conditions of Theorem \ref{P1 solve suff} are satisfied. Hence, the input defined as \eqref{u input form} satisfies \eqref{u const P1}. 
\end{proof}

The constraints of the QP \eqref{QP u P2} change at time instant $t_i$ for $1\leq 1\leq N$. Note that we assume that the functions $h(x), h_i(x)$ are continuously differentiable to be able to use \eqref{u const P1} or \eqref{QP u P2}. In Section \ref{sec: discussion}, we discuss how to overcome this limitation. 
% This gives a systematic way of solving Problem \ref{P S Si}. 

\section{Simulations}\label{sec: simulations}
We present three numerical examples to demonstrate the efficacy of the proposed methods. In the first scenario, we consider the example of reaching a set $S_1$ in a prescribed time $T$, and stay there for all future times, while also remaining in a $S_2$ at all times. Mathematically, the closed-loop trajectories are required to satisfy $x(t) \in S_1, \; \forall t\geq T,\quad x(t) \in S_2, \; \forall t\geq 0$, with $x(0)\in S_2$. The system dynamics are considered as
\begin{align*}
    \dot x_1 & = -x_2 + x_1^2 + x_1u\\
    \dot x_2 & = x_1 +x_2\tanh x_2+x_2u,
\end{align*}
where the state-vector is $x = [x_1 \; x_2]^T\in \mathbb R^2$ and the control input is $u\in \mathbb R$. Note that the open-loop trajectories for these dynamics diverge to infinity, i.e., the origin is unstable for the open-loop system. We choose $S_1 = \{x\; |\; \|x\|\leq 1\}$ and $S_2 = \{x\; |\; \frac{x_1^2}{9^2} + \frac{x_2^2}{0.9^2} \leq 1\}$ and $T = 10$ sec. Figure \ref{result:01} shows the closed-loop trajectories for four different initial conditions. The trajectories reach the set $S_1$ in  prescribed time and stay there at all the future times, while remaining in the set $S_2$ at all times.

% \begin{figure}[!ht]
%     \centering
%         \includegraphics[width=0.95\columnwidth,clip]{scenario_01.eps}
%     \caption{Scenario 1: Set $S_1$ and $S_2$.}\label{scene:01}
% \end{figure}

\begin{figure}[!ht]
    \centering
        \includegraphics[width=0.95\columnwidth,clip]{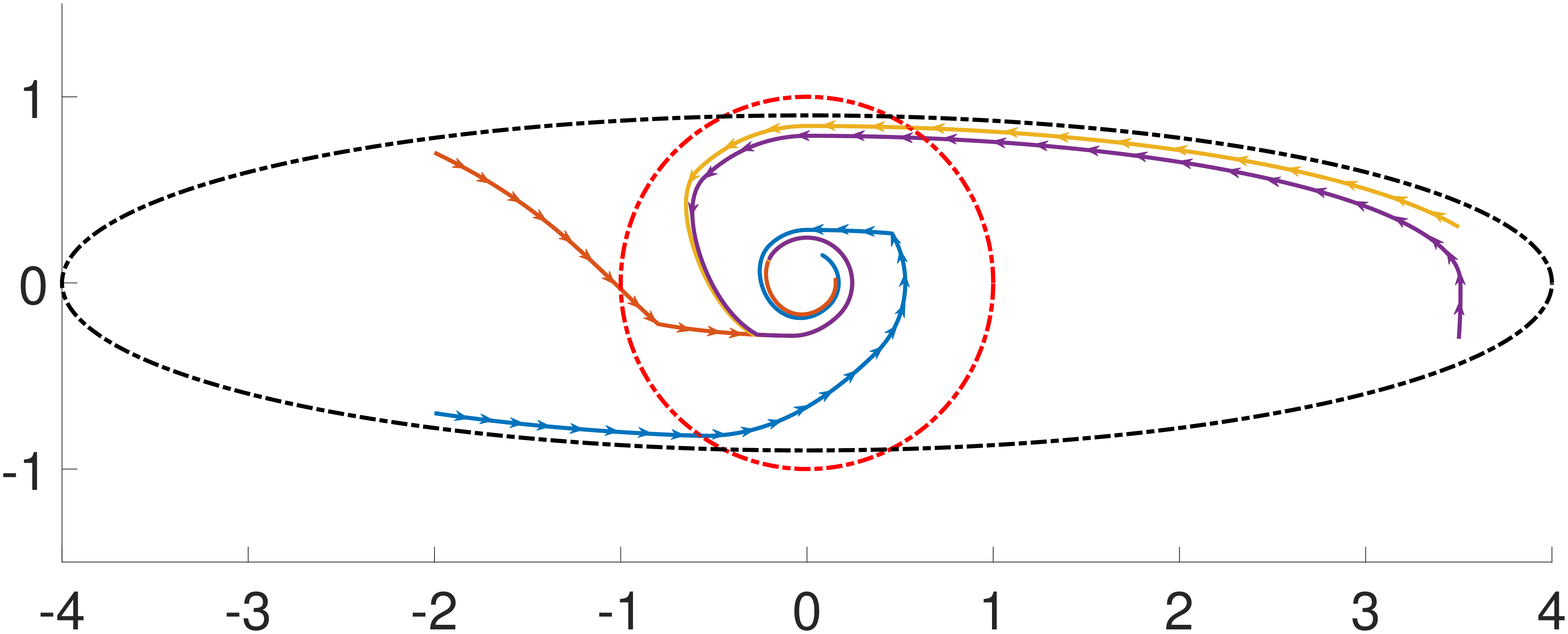}
    \caption{Scenario 1: Closed-loop trajectories.}\label{result:01}
\end{figure}

In second scenario, we take Example 2 from \cite{lindemann2019control} and use our proposed method to satisfy the STL specifications $\phi= \phi_1\land\phi_2$, where $\phi_1 = F_{[5,15]}(\|x-[10 \; 0]^T\|\leq 5)$ and $\phi_2 = G_{[5,15]}(\|x-[10 \; 5]^T\|\leq 10)$, with $S_1 = \{x\; |\; \|x-[10 \; 5]^T\|\leq 10\}$ and $S_2 = \{x\; |\; \|x-[10 \; 0]^T\|\leq 5\}$. The robot dynamics are modeled as $\dot x = u$ where $x,u\in \mathbb R^2$. We use $\|u\|\leq 10$ as the control input constraints. In order to translate the input constraint in the form of \eqref{input bound const}, we define $A_u = \begin{bmatrix}1& 0 \\ -1& 0\\  0 & 1\\ 0 &-1\end{bmatrix}$ and $b_u = \begin{bmatrix}7&7&7&7\end{bmatrix}^T$, so that $u_{x}, u_{y}\in [-7, \; 7]$. Figure \ref{scene:0 traj} shows the closed-loop trajectories for various initial conditions outside the set $S_1$. It can be seen that the trajectories reach the set $S_1$ and stay in $S_1$ at all future times, and then reach set $S_2$. Figure \ref{scene:0 u} shows the norm of the control input $u(t)$ with time. As can be seen from the figure, the control input jumps at $t = 5$ sec, when the system trajectories reach the set $S_1$. 

\begin{figure}[!ht]
    \centering
        \includegraphics[ width=0.85\columnwidth,clip]{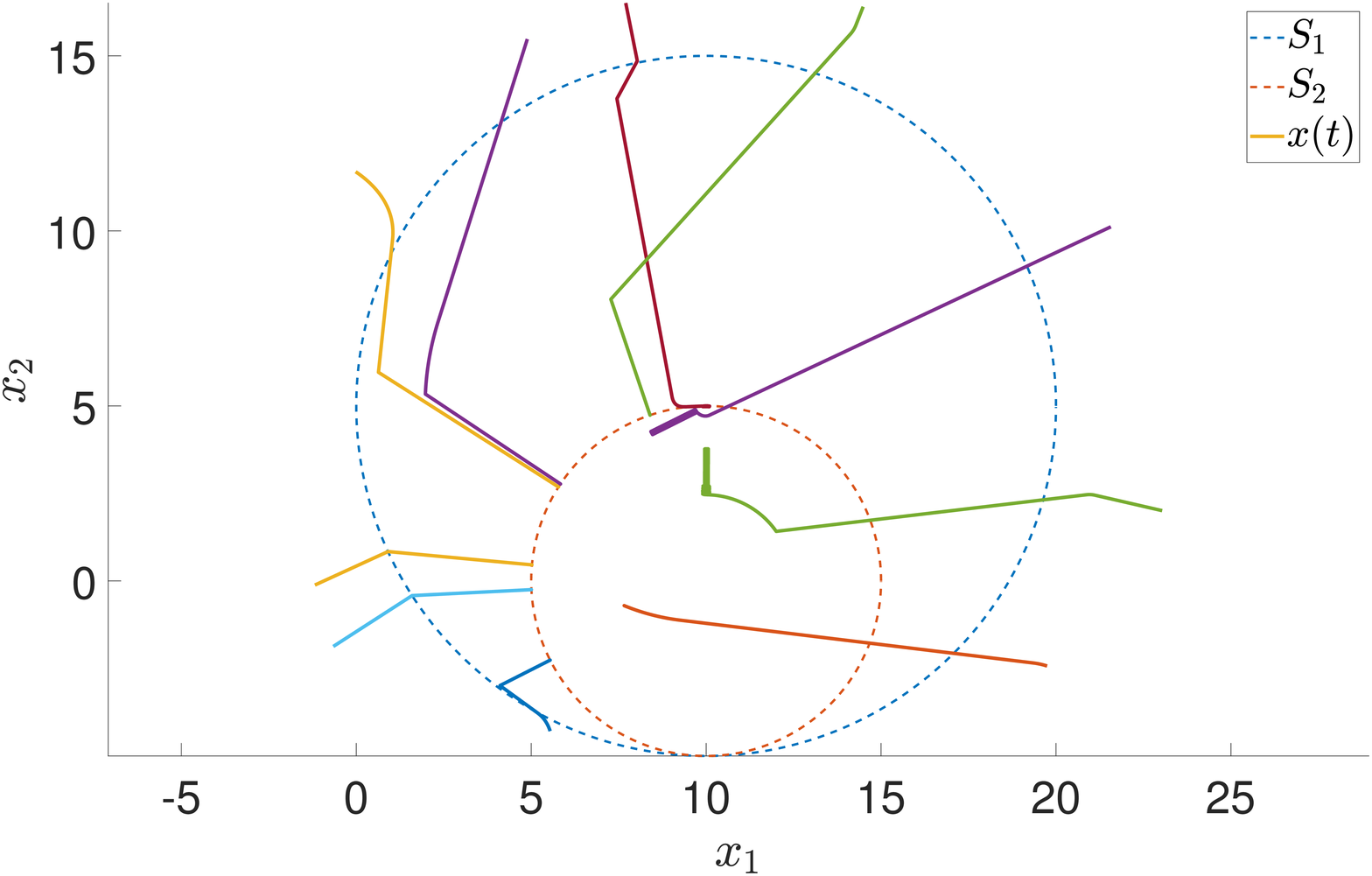}
    \caption{Scenario 2: Closed-loop trajectories for various initial conditions.}\label{scene:0 traj}
\end{figure}

\begin{figure}[!ht]
    \centering
        \includegraphics[ width=0.85\columnwidth,clip]{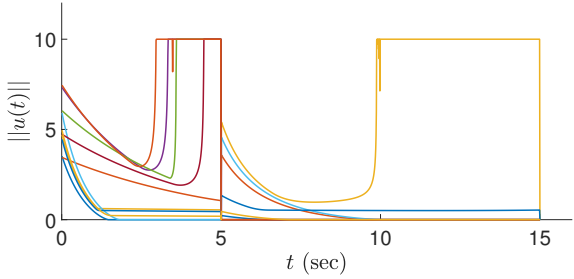}
    \caption{Scenario 2: Control input for various initial conditions.}\label{scene:0 u}
\end{figure}

In the third scenario, we present a method of constructing sets $S_i$ for applications such as robot motion planning, where the conditions of Theorem \ref{P1 solve suff} are not met. The closed-loop trajectories, starting from $x(0)\in C_1$, are required to satisfy the following spatio-temporal specifications$ (x_1,t) \models G_{[0,T_4]}\phi_s\land F_{[0, T_1]}\phi_2\land F_{[T_1, T_2]}\phi_3\land F_{[T_2, T_3]}\phi_4\land F_{[T_3, T_4]}\phi_1$, 
which is explained in details below (see Figure \ref{scene:1}):
 
\begin{figure}[!ht]
    \centering
        \includegraphics[ width=0.8\columnwidth,clip]{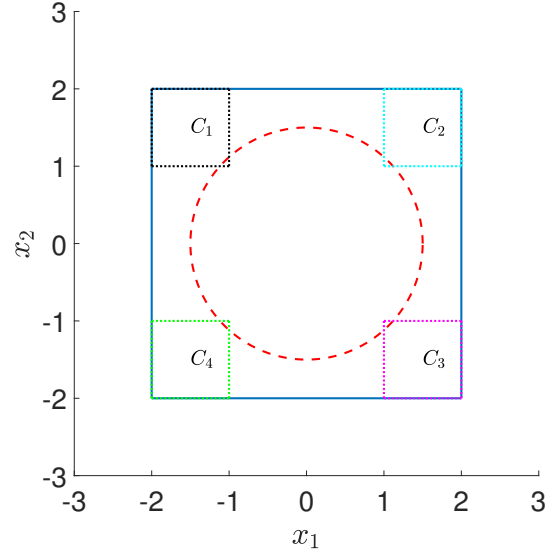}
    \caption{Scenario 3: Problem setting.}\label{scene:1}
\end{figure}

\begin{itemize}
    \item $x(t)\in S_s = \{x\; |\; \|x\|_1\leq 2 \bigcap\|x\|_2\geq 1\}$ for all $t\geq 0$, i.e., the closed-loop trajectories should stay inside the solid-blue square and outside the red-dotted circle at all times;
    \item Before a given $0<T_1<\infty$, $x(T_1)\in C_2 = \{x\; |\; \|x-[1.5 \; 1.5]^T\|_1\leq 0.5\}$;
    \item Before a given $T_1<T_2<\infty$, $x(T_2)\in C_3  = \{x\; |\; \|x-[1.5 \; -1.5]^T\|_1\leq 0.5\}$;
    \item Before a given $T_2<T_3<\infty$, $x(T_3)\in C_4 =  \{x\; |\; \|x-[-1.5 \; -1.5]^T\|_1\leq 0.5\}$;
    \item Before a given $T_3<T_4<\infty$, $x(T_4)\in C_1 = \{x\; |\; \|x-[-1.5 \; 1.5]^T\|_1\leq 0.5\}$.
    % \item $|u(t)|\leq 10$ for all $t\geq 0$. 
\end{itemize}

This problem is an extended version of the case study considered in \cite{lindemann2017robust}. Note that the sets $C_i$ are not overlapping with each other, and the corresponding functions $h_i(x)$ are not continuously differentiable. 
Now, in order to be able to use QP-based formulation \eqref{QP u P2}, we need to find the sets $\bar S_i$ such that $\bar S_i\bigcap \bar S_{i+1}\neq\emptyset$. 

\begin{figure}[!htbp]
    \centering
        \includegraphics[ width=0.8\columnwidth,clip]{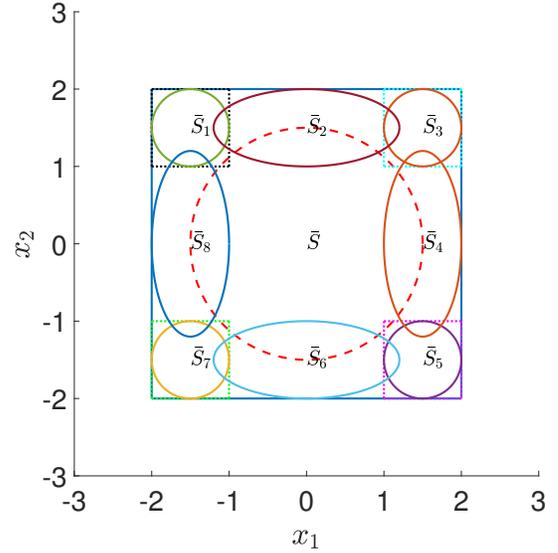}
    \caption{Scenario 3: Construction of sets $\bar S, \bar S_1, \cdots, \bar S_8$.}\label{scene:2}
\end{figure}

The set $\bar S = \{x\; |\; \|x\|\leq 1.5\}$ and sets $\bar S_i$ are defined as follows (see Figure \ref{scene:2}):
\begin{itemize}
    \item $\bar S_1 = \{x\; |\; \|(x-[-1.5 \; 1.5]^T)\|\leq 1\}$;
    \item $\bar S_2 = \{x\; |\; \|(x-[0 \; 1.5]^T)\|_{P_1}\leq 1\}$;
    \item $\bar S_3 = \{x\; |\; \|(x-[1.5 \; 1.5]^T)\|\leq 1\}$;
    \item $\bar S_4 = \{x\; |\; \|(x-[1.5 \; 0]^T)\|_{P_2}\leq 1\}$;
    \item $\bar S_5 = \{x\; |\; \|(x-[1.5 \; -1.5]^T)\|\leq 1\}$;
    \item $\bar S_6 = \{x\; |\; \|(x-[0 \; -1.5]^T)\|_{P_1}\leq 1\}$;
    \item $\bar S_7 = \{x\; |\; \|(x-[-1.5 \; -1.5]^T)\|\leq 1\}$;
    \item $\bar S_8 = \{x\; |\; \|(x-[-1.5 \; 0]^T)\|_{P_2}\leq 1\}$;
\end{itemize}
where $\|z\|_{P_1} = \sqrt{\frac{z_1^2}{1.2^2}+\frac{z_2^2}{0.5^2}}$ and $\|z\|_{P_2} = \sqrt{\frac{z_1^2}{0.5^2}+\frac{z_2^2}{1.2^2}}$. The problem can be re-formulated to design a control input $u(t)$ such that for $x(0)\in \bar S_1$, 
\begin{itemize}
    \item For a given $0<t_0<T_1$, $x(t_0)\in \bar S_2\setminus \bar S$;
    \item For a given $t_0<t_1\leq T_1$, $x(t_1)\in \bar S_3\setminus \bar S$;
    \item For a given $T_1<t_2<T_2$, $x(t_2)\in \bar S_4\setminus \bar S$;
    \item For a given $t_2<t_3\leq T_2$, $x(t_3)\in \bar S_5\setminus \bar S$;
    \item For a given $T_2<t_4<T_3$, $x(t_4)\in \bar S_6\setminus \bar S$;
    \item For a given $t_4<t_5\leq T_3$, $x(t_5)\in \bar S_7\setminus \bar S$;
    \item For a given $T_3<t_6<T_4$, $x(t_6)\in \bar S_8\setminus \bar S$;
    \item For a given $t_6<t_7\leq T_4$, $x(t_7)\in \bar S_1\setminus \bar S$,
\end{itemize}
which can be written as an STL formula as in \eqref{eq: robot STL req}.

\begin{figure*}[!ht]
% ensure that we have normalsize text
% \small
\begin{align}\label{eq: robot STL req}
(x_1,t) & \models G_{[0,T_4]} \bar \phi_s\land G_{[0, t_0]}\bar \phi_1\land F_{[0, t_0]}\bar \phi_2\land  G_{(t_0,T_1]}\bar \phi_2\land F_{(t_0, T_1]}\bar \phi_3\land G_{(T_1,t_2]}\bar \phi_3\land F_{(T_1, t_2]}\bar \phi_4 \land G_{(t_2,T_2]}\bar \phi_4\land F_{(t_2, T_2]}\bar \phi_5\nonumber\\
& \land G_{(T_2,t_4]}\bar \phi_5\land F_{(T_2, t_4]}\bar \phi_6\land G_{(t_4,T_3]}\bar \phi_6\land F_{(t_4, T_3]}\bar \phi_7\land G_{(T_3,t_6]}\bar \phi_7\land F_{(T_3, t_6]}\bar \phi_8\land G_{(t_6,T_4]}\bar \phi_8\land  F_{(t_6, T_4]}\bar \phi_1.
\end{align}
% \end{subequations}

% Restore the current equation number.
% \setcounter{equation}{\value{MYtempeqncnt}}
\hrulefill
\end{figure*}

We can now use the formulation \eqref{QP u P2} to compute the control input. We use $|u|\leq 10$ as the input constraints and use the same approach as in scenario 2, to translate these constraints in the form of \eqref{input bound const}. The time constraints are chosen as $T_i = 2$ for $i \in \{1,2,3,4\}$ and $t_j = 1$ for $j\in \{t_0, t_1,\cdots ,t_7\}$. We choose $\mu = 5$, so that $\gamma_1 = 1.2$ and $\gamma_2 = 0.8$. Figures \ref{traj:1 1}-\ref{traj:1 2} illustrate the closed-loop position trajectories of the robot for one initial condition; it is evident that the robot position always remains in the safe set $S_s$, while visiting the sets $C_2, C_3, C_4$ and $C_1$ sequentially.Figure \ref{fig: ex3 u1 u2 dist} illustrates the control input trajectories and verifies that the control input constraint $\|u_i(t)\|\leq 10$ is satisfied at all times.

\begin{figure}[!ht]
	\centering
% 	\begin{subfigure}{0.45\textwidth} % width of left subfigure
		\includegraphics[width=0.4\textwidth]{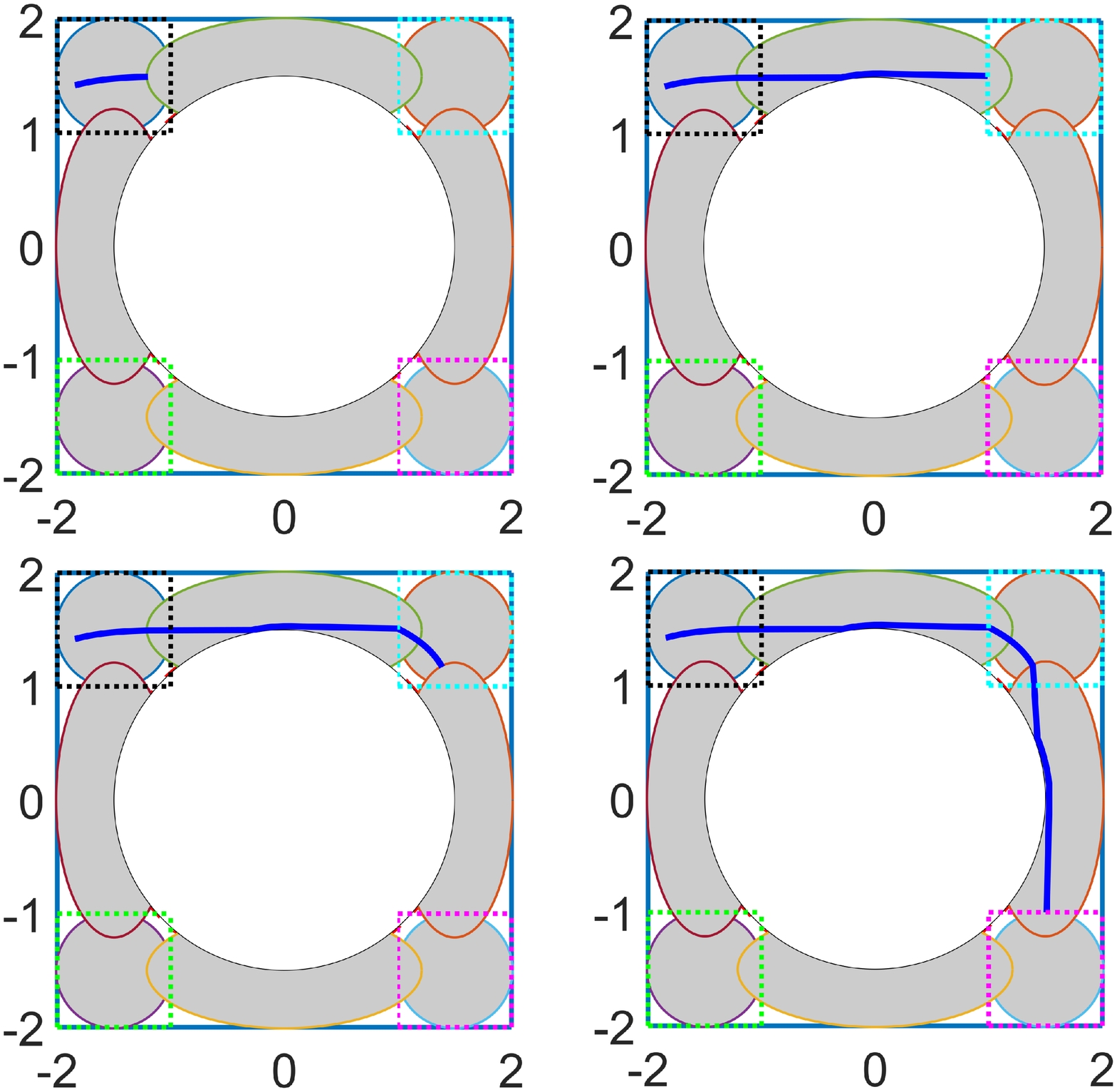}
		\caption{Scenario 3: Closed-loop trajectory: snapshot at $t = 1,2,3$ and $4$ sec.}\label{traj:1 1}  % subcaption
% 	\end{subfigure}
	\vspace{1em} % here you can insert horizontal or vertical space
% 	\begin{subfigure}{0.45\textwidth} % width of right subfigure
		\includegraphics[width=0.4\textwidth]{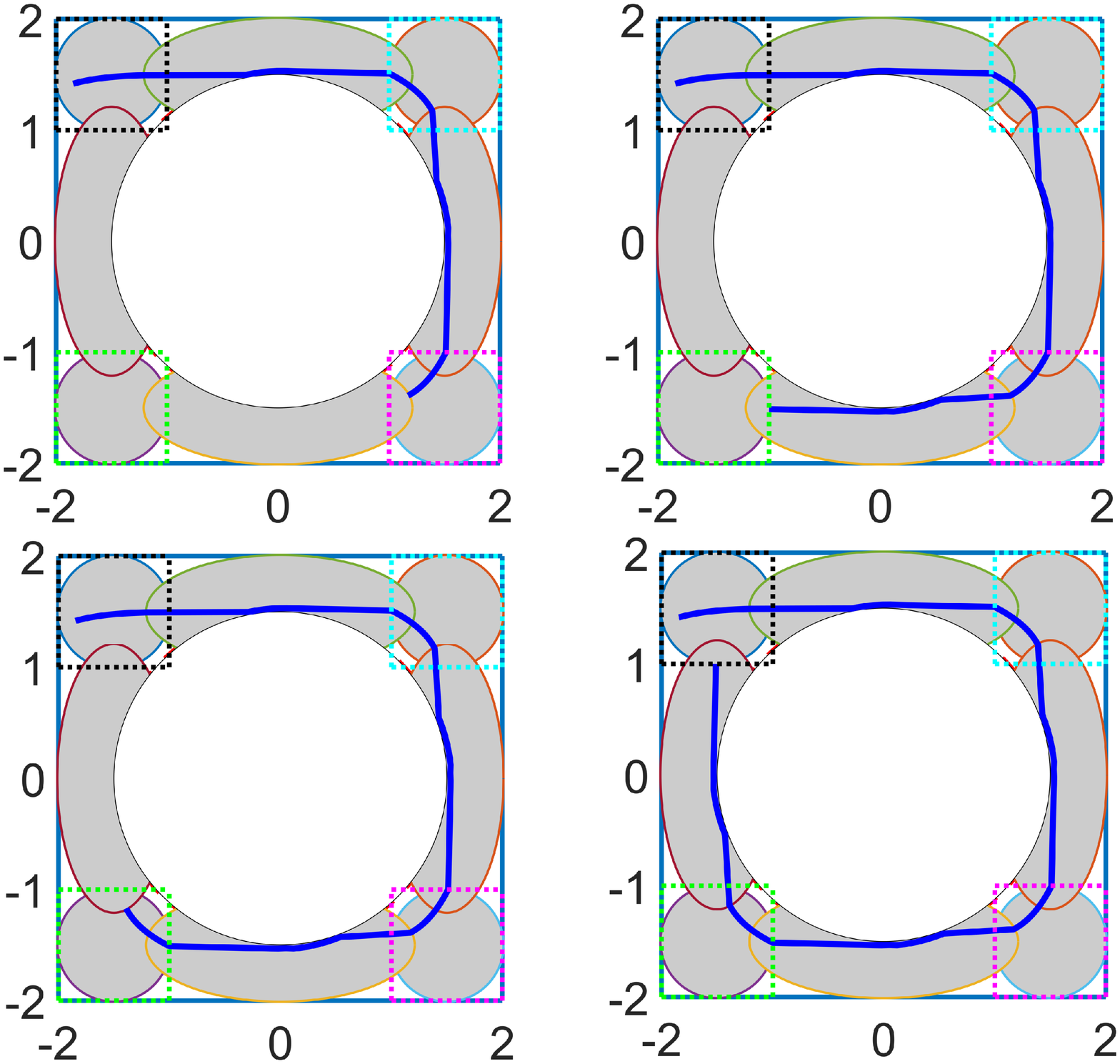}
		\caption{Scenario 3: Closed-loop trajectory: snapshot at $t = 5,6,7$ and $8$ sec.} \label{traj:1 2} % subcaption
% 	\end{subfigure}
\end{figure}

\begin{figure}[!ht]
    \centering
    \includegraphics[ width=1\columnwidth,clip]{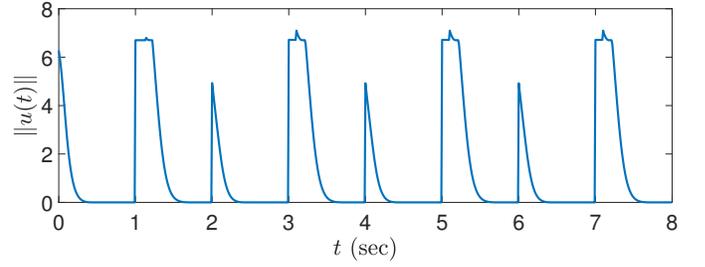}
    \caption{Scenario 3: Norm of the control input $\|u(t)\|$.}\label{fig: ex3 u1 u2 dist}
\end{figure}

% \begin{figure}[!ht]
%     \centering
%         \includegraphics[ width=0.95\columnwidth,clip]{traj_1_1.eps}
%     \caption{Closed-loop trajectories for scenario 2.}\label{traj:1}
% \end{figure}

\section{Direction for future work}\label{sec: discussion}
Note that Theorem \ref{FxTS reach set S}, Corollary \ref{Set S reach and inv} and Theorem \ref{P1 solve suff} are restrictive because of the following two reasons. First, it is needed that the functions $h(x)$ and $h_i(x)$ are CLF/CBF for the system \eqref{cont aff sys}, otherwise the respective inequalities used in the aforementioned results do not hold. Second, these results also need that the functions $h(x)$ and $h_i(x)$ are continuously differentiable. Although, this is a very common assumption in the literature (see \cite{li2018formally, ames2017control} and other similar work), this limits the choice of sets $S$ and $S_i$ that can be considered in the setup of Theorem \ref{FxTS reach set S} or Theorem \ref{P1 solve suff}. One approach is to use non-smooth analysis (e.g., \cite{sontag1995nonsmooth}) to formulate the constraints of \eqref{QP u P2}, so that the sets characterized by non-differentiable $h(x)$ can also be incorporated. Another plausible approach is to look for the CLF and the control input $u$ simultaneously. We propose sufficient conditions to characterize the CLF and the control input for Problem \ref{P reach S}.  

\begin{Proposition}\label{CLF relaxation thm}
If there exist continuously differentiable function $V$ and constants $a_1, a_2>0$, $\gamma_1>1$ and $0<\gamma_2<1$ satisfying $\frac{1}{a_1(\gamma_1-1)} + \frac{1}{a_2(1-\gamma_2)}\leq T$ such that the following holds for $x\notin S$
\begin{subequations}
\begin{align}
    h(x)\leq V(x) \leq &h(x)+c \label{hx Vx ineq}\\
    \inf_{u\in \mathcal U}\{L_fV(x) + L_gV(x)u & +a_1h(x)^{\gamma_1}+a_2h(x)^{\gamma_2}\}\leq 0, \label{fxts h relax ineq}
\end{align}
\end{subequations}
where $c\geq 0$, then the closed-loop trajectories of \eqref{cont aff sys} reach the set $S$ within prescribed time $T$ for all initial conditions.
\end{Proposition}
\begin{proof}
If $h(x)$ is a smooth function, one can choose $c = 0$, so that $V(x) = h(x)$. Note that \eqref{hx Vx ineq} implies that $V(x)\leq h(x)$ for $x\notin S$. If there exists a control input $u$, such that \eqref{fxts h relax ineq} holds, then from \eqref{hx Vx ineq}, we obtain that 
\begin{align*}
    \dot V+a_1V(x)^{\gamma_1}+a_2V(x)^{\gamma_2} &\leq \dot V+ a_1h(x)^{\gamma_1}+a_2h(x)^{\gamma_2}\\
    &\leq 0.
\end{align*}
Hence, using Theorem \ref{FxTS TH}, we obtain that $V(x(t)) =0$ for all $t\geq \bar T$, where $\bar T\leq  \frac{1}{a_1(\gamma_1-1)} + \frac{1}{a_2(1-\gamma_2)} \leq T$. Now, from \eqref{hx Vx ineq}, we know that for $V(x) = 0 \implies h(x)\leq 0$, which implies $x\in S$. 
\end{proof}
We illustrate, via a simple example, how Proposition \ref{CLF relaxation thm} can be used for the case when $h(x)$ is non-smooth. Consider the case when the set $S$ in Problem \ref{P reach S} is defined as $S = \{x\; |\; \|x\|_1\leq 1\}$, i.e., using the 1-norm of $x$, so that $S$ is a square. Using the fact that $S_n = \{x\; |\; x_1^{2n}+x_2^{2n}-1\} \rightarrow S$, as $n\rightarrow \infty$, one can choose $V = x_1^{2n} + x_2^{2n}-1$, for large positive integer $n$ and look for $n$, along with $u, a_1, a_2>0$, $\gamma_1>1$ and $0<\gamma_2<1$ so that conditions of Proposition \ref{CLF relaxation thm} hold. A similar set of sufficient conditions can be derived for Theorem \ref{P1 solve suff}, which would allow a larger class of problems to be solved. It is part of our future investigations to study methods to solve for $V$ and $u$, simultaneously, in an efficient way. In future, we would also like to study properties of the system dynamics and the functions $h(x), h_i(x)$, so that the resulting closed-loop trajectories are smooth.  

\section{Conclusions}\label{sec: conclusion}
In this paper, we considered the problem of trajectory planning under spatio-temporal, and control input constraints. We defined a new class of CLF, called PT CLF, to guarantee that the closed loop trajectories reach a given set within the prescribed time. We formulated a QP to find a control input that guarantees prescribed time convergence. Then, we considered a general problem of control synthesis under multiple spatiotemporal objectives. We first presented sufficient conditions for the existence of a control input in terms of PT CLF and CBF. Then, we presented a QP based formulation to efficiently compute the control input that guarantees safety and prescribed time convergence in the presence of control input constraints.

\bibliographystyle{IEEEtran}
\bibliography{myreferences}

\end{document}